\documentclass{amsart}

\usepackage{amssymb} \usepackage{amsfonts} \usepackage{amsmath}
\usepackage{amsthm} 
\usepackage{epsfig} 
\usepackage{color}
\usepackage{amscd}
\usepackage[all]{xy}
\usepackage[english]{babel}
\usepackage{todonotes}
\usepackage[capitalize]{cleveref}
\usepackage{amssymb, amsfonts, amsmath}
\usepackage{pgfplots}
\usepackage{comment} 
\usepackage{mathtools}
\usepackage{wrapfig}
\usepackage{tikz}
\usepackage{mathrsfs}
\usepackage{tabularx}
\usepackage{amssymb} \usepackage{amsfonts} \usepackage{amsmath}
\usepackage{amsthm} \usepackage{epsfig, subfig}
\usepackage{ amscd, amsxtra, latexsym}
\usepackage{epsfig,  graphicx, psfrag}
\usepackage[all]{xy}
\usepackage{caption}
\usepackage{enumerate}
\usepackage{color}
\pgfplotsset{compat=1.13} 

\usepackage{esint}
\usepackage{array}
\usepackage [english]{babel}
\usepackage [autostyle, english = american]{csquotes}
\MakeOuterQuote{"}

\newtheorem{lemma}{Lemma}[section]
\newtheorem{thm}[lemma]{Theorem}
\newtheorem{prop}[lemma]{Proposition}
\newtheorem{cor}[lemma]{Corollary}

\newtheorem{prop_intro}{Proposition}

\newtheorem{thm_intro}[prop_intro]{Theorem}

\newtheorem{cor_intro}[prop_intro]{Corollary}

\theoremstyle{definition}

\newtheorem{defn}[lemma]{Definition}

\newtheorem{rem}[lemma]{Remark}

\newtheoremstyle{citing}
{3pt}
{3pt}
{\itshape}
{}
{\bfseries}
{.}
{.5em}
{\thmnote{#3}}
\theoremstyle{citing}

\newcommand{\calT} {\ensuremath {\mathcal{T}}}

\newcommand{\R}{\mathbb{R}}
\newcommand{\C}{\mathbb{C}}

\newcommand{\Hp} {\ensuremath {\mathbb{H}^2}}
\newcommand{\Hn} {\ensuremath {\mathbb{H}^n}}

\newcommand{\str} {\ensuremath {\textrm{str}}}
\newcommand{\strtil} {\ensuremath {\widetilde{\textrm {str}}}}

\newcommand{\bi} {{\partial}}

\DeclareMathOperator{\isom}{Isom}

\newcommand{\commento}[1]{}

\author{Gian Maria Dall'Ara}
\address{Istituto Nazionale di Alta Matematica ``Francesco Severi", Research Unit Scuola Normale Superiore, 
	Piazza dei Cavalieri 7, 56126 Pisa, Italy}
\email{dallara@altamatematica.it}

\author{Roberto Frigerio}
\address{Dipartimento di Matematica, Largo Pontecorvo 5, 56127 Pisa, Italy}
\email{roberto.frigerio@unipi.it}

\author{Ervin Had\v{z}iosmanovi\'c}
\address{Scuola Normale Superiore, 
Piazza dei Cavalieri 7, 56126 Pisa, Italy}
\email{ervin.hadziosmanovic@sns.it}

\title[]{Bounded cohomology classes \\ from differential forms}

\begin{document}
	
\begin{abstract}
Let $M$ be a complete hyperbolic $n$-manifold, $n\geq 2$. Via integration over geodesic simplices, any closed bounded differential 2-form on $M$ 
defines a bounded cohomology class in $H^2_b(M)$. 
It was proved by Barge and Ghys (for $n=  2$) and by Battista et al.~(for $n>2$) that, if $M$ is closed, then this procedure 
defines an injective embedding of the (infinite-dimensional) space of closed  differential $2$-forms
on $M$ into  $H^2_b(M)$. 

We extend this result to the case when the fundamental group of $M$ is of the first kind, i.e.~its limit set is equal to the whole
boundary at infinity of hyperbolic space (this holds, for example, when $M$ has finite volume).
Our argument is different from Barge and Ghys' original one, and relies on the following fact of independent interest: an $L^\infty$ function on the hyperbolic plane is determined by its integrals over all ideal triangles. We prove this fact by way of Fourier analysis on the hyperbolic plane. 

\end{abstract}

\maketitle

\section{Introduction}
Bounded cohomology is a functional-analytic analogue of singular cohomology introduced by Johnson \cite{johnson1972cohomology} and Trauber for Banach algebras and defined by Gromov for topological spaces and groups in his seminal paper \cite{Gromov82}. Due to the lack of the excision property, bounded cohomology is very different from ordinary cohomology. Roughly speaking, bounded cohomology is usually large (i.e.~infinite dimensional) in presence of negative
curvature, while it vanishes for non-negatively curved spaces. For example, the bounded cohomology of spheres and of $n$-dimensional tori (and, more in general, of any
topological space with an amenable fundamental group) vanishes in every positive degree, 
while it is infinite-dimensional in degree 2 and 3 for closed hyperbolic manifolds of any dimension bigger than one. In fact, the bounded cohomology of negatively curved manifolds has been extensively studied, both via techniques coming from (geometric) group
theory  (e.g.~via the theory of quasimorphisms and quasicocycles~\cite{BrooksSeries, Picaud, Fuji1, Calegari, extension, zeronorm}), and via the study of bounded classes coming from  differential forms (see e.g.~\cite{bargeghys, soma1, soma2,  farre, Marasco, Battista, Bargagnati}). 

In this paper we focus on the study of the second bounded cohomology group of hyperbolic manifolds via integration of differential forms. 
Let $M$ be a complete hyperbolic manifold and let $\Omega^2(M)$ be the space of smooth differential $2$-forms on $M$. 
In 1988, Barge and Ghys~\cite{bargeghys} proved that, if $\dim M=2$ and $M$ is closed (i.e.~compact without boundary), then
the map
\begin{equation}\label{psi}
\psi\colon \Omega^2(M)\to H^2_b(M),
\end{equation}
defined by integrating differential forms over geodesic triangles, is injective (see \cref{sec:preliminari} for the precise definition), thus showing that $H^2_b(M)$ contains
an infinite-dimensional subspace represented by differential forms. Very recently, in~\cite{Battista} 
it was shown that Barge and Ghys' argument 
also applies to degree-$2$ closed differential forms in higher dimensional {closed} manifolds.

A natural question is whether the injectivity of $\psi$ may hold even without the assumption that $M$ be compact. 
First observe that, if $M$ is compact, then cocycles corresponding to smooth differential $2$-forms are indeed bounded because geodesic triangles in $M$ have uniformly bounded area, and any differential form on a compact manifold is bounded. When $M$ is non-compact, in order to define {the map $\psi$} we need to restrict to bounded forms. 
Recall that a degree-$k$ form  on  $M$ is bounded if the supremum of the values it attains on all the
orthonormal $k$-frames of the tangent bundle of $M$ is bounded. If $Z\Omega^2_b(M)$ denotes the space of closed and bounded 
differential $2$-forms on $M$, then we obtain a map
$$
\psi\colon Z\Omega^2_b(M)\to H^2_b(M)\ .
$$
As mentioned above, it was shown in~\cite{Battista} that this map is injective provided that $M$ be closed. If $M$ is non-compact, then
it may well be that $\psi$ is not injective: this happens, for example, if the fundamental group of $M$ is amenable
in which case
 $H^2_b(M)=0$, while of course $ Z\Omega^2_b(M)\neq 0$. Our main theorem completely characterizes the injectivity of $\psi$
 in terms of the dynamics of the fundamental group of $M$ on the boundary at infinity of hyperbolic space $\mathbb{H}^n$. If $\Gamma$ is a 
 subgroup of $\isom(\Hn)$, then its \emph{limit set} $\Lambda(\Gamma)$ is the intersection between $\bi\Hn$ and  the closure
 in $\overline{\Hn}$ of any $\Gamma$-orbit in $\Hn$ (it is well known and easy to check that such intersection does not depend on the chosen $\Gamma$-orbit). 
 Recall that a complete hyperbolic manifold $M=\mathbb{H}^n/\Gamma$ (or its fundamental group $\Gamma$) is \emph{of the first kind} if $\Lambda(\Gamma)$ is equal to the whole $\partial \mathbb{H}^n$, and it is \emph{of the second kind} otherwise.
  In this paper we prove the following:

\begin{thm_intro}\label{main:thm}
Let $M$ be a complete hyperbolic $n$-manifold, $n\geq 2$. Then the map
$$
\psi\colon Z\Omega^2_b(M)\to H^2_b(M)
$$
is injective if and only if $M$ is of the first kind.
\end{thm_intro}

It is well known that complete finite volume hyperbolic manifolds of any dimension are of the first kind, hence we get the following:
 
 \begin{cor_intro}
 Let $M$ be a complete finite volume hyperbolic manifold. Then the map
 $$
\psi\colon Z\Omega^2_b(M)\to H^2_b(M)
 $$
is injective.
 \end{cor_intro}

If $M$ is not of the first kind, then the convex core $C$ of $M$
is a proper subset of $M$, and it is not difficult to show that any differential $2$-form supported in the complement of $C$ belongs to the kernel of $\psi$ 
(see Section~\ref{noninj:sec}). 
Thus, the ``only if'' implication of Theorem~\ref{main:thm} is quite elementary.

On the other hand, proving the injectivity of $\psi$ under the assumption that $M$ is of the first kind is much harder.
The arguments described in~\cite{bargeghys} (and in~\cite{Battista}, for the higher dimensional case) do not seem to generalize to the case of  
non-compact manifolds. 
In fact, Barge--Ghys' argument relies on the fact that a differential $2$-form $\omega$ such that $\psi(\omega)=0$ admits a primitive $\alpha$ whose integral on every
closed geodesic vanishes.
Then, Liv\v{s}ic's Theorem for the geodesic flow on the unit tangent bundle~\cite[Corollary 1.5]{Croke}  implies that $\alpha$ is exact, which gives
in turn that $\omega=d\alpha=0$.
Problematic issues seem to arise 
if we allow $M$ to be non-compact:
while a version of Liv\v{s}ic's Theorem exists at least for the geodesic flow on finite volume hyperbolic manifolds \cite[Theorem 1.3]{ramirez2013invariant}, such result applies only to \textit{bounded} differential forms, and it is easy to construct examples of exact bounded  differential $2$-forms which do not admit any bounded primitive,
even for finite area surfaces (the existence of a bounded primitive of the volume form implies the positivity of Cheeger's isoperimetric constant \cite[Proposition 2.6]{ervin_cheeger}, but this constant vanishes, for example, for once-punctured tori endowed with a complete, finite area hyperbolic metric). 
Also observe that, for complete hyperbolic manifolds, the ergodicity of the geodesic flow (which may prove useful in arguments \emph{\`a la Liv\v{s}ic}) is in general a strictly stronger 
condition than being of the first
kind: 
there exist hyperbolic surfaces (of infinite type) of the first kind 
whose geodesic flow is not ergodic (see e.g.~\cite{basmajian2022type}  and \cite{pandazis2023ergodicity}).  We refer the reader to Remark~\ref{dynamic:rem} below for more
details on the relationship between the ergodocity of the geodesic flow of $M=\mathbb{H}^n/\Gamma$ and the dynamics of the action of $\Gamma$ on $\bi\Hn$.

 

In order to prove our main result, we follow a completely different strategy, which is based on the computation of bounded cohomology via resolutions coming from amenable actions.
Roughly speaking, amenable actions play in the theory of bounded cohomology an analogous r\^ole as  properly discontinuous actions in the theory of classical cohomology. In our context, thanks to the amenability of the action of $\Gamma$ on $\Hn$ (see Lemma~\ref{amenability} {below}), and a fundamental result by Burger and Monod  (Theorem~\ref{amenable_resolution} below)
we can compute the bounded cohomology of $M=\mathbb{H}^n/\Gamma$ by looking at the cochain complex  $L^\infty_{\rm alt} ((\partial \Hn)^{\bullet +1})^\Gamma$ of
$\Gamma$-invariant 
alternating $L^\infty$-functions on $(\partial \Hn)^{\bullet+1}$.

For $\omega\in Z\Omega^2_b(M)$,  the class $\psi(\omega)$ is represented in $L^\infty_{\rm alt} ((\partial\Hn)^3)^\Gamma$ by the cocycle
which associates to the triple $(\xi_0,\xi_1,\xi_2)\in (\partial \Hn)^3$ the integral of the lift of $\omega$ on the ideal triangle with vertices $\xi_0,\xi_1,\xi_2$. Now,
if we assume the action of $\Gamma$ on $\partial \Hn$ to be doubly ergodic (which is well known to be equivalent to the ergodicity of the geodesic flow of $M$, see Remark~\ref{dynamic:rem} below), 
the cochain complex  $L^\infty_{\rm alt} ((\partial \Hn)^{\bullet +1})^\Gamma$ does not contain any non-trivial coboundary in degree 2, hence 
$\psi(\omega)$ is trivial in bounded cohomology if and only if the cocycle just described vanishes identically. In fact, thanks to~\cite[Proposition 3.1]{BurgerIozzi},
a refined version of this argument works also under the weaker assumption that $\Gamma$ be of the first kind.
Since the lift $\widetilde{\omega}$ of $\omega$ to $\mathbb{H}^n$ is clearly bounded, in order to conclude
the proof of Theorem~\ref{main:thm} we are thus left to show that the following holds:

\begin{thm_intro}\label{thm:pompeiu_1}
Let $\widetilde{\omega}\in Z\Omega^2_b(\mathbb{H}^n)$ be such that
$$
\int_T \widetilde{\omega}=0
$$
for every ideal triangle $T\subseteq \Hn$. Then $\widetilde{\omega}=0$.
\end{thm_intro}

It is easy to reduce Theorem~\ref{thm:pompeiu_1} to the case $n=2$. If $n=2$,
we have $\widetilde{\omega}=f\, dA$ for some smooth bounded function $f\colon \Hn\to\mathbb{R}$, where $dA$ is the hyperbolic area form. 
Thus, we will finally deduce (the difficult implication of) Theorem~\ref{main:thm} from  the following:

\begin{thm_intro}\label{thm:pompeiu_0}
Let $f\in L^\infty (\Hp)$ be such that
$$
\int_T f\, dA=0
$$
for every ideal triangle $T\subseteq \Hp$, where $dA$ is the hyperbolic area. Then $f=0$.
\end{thm_intro}

This is a new result in integral geometry, the field centered around the problem of reconstructing a function defined on a manifold (or, more generally, a section of a vector bundle) from its integrals over a given family of submanifolds (see \cite{gelfand_book} for an introduction to this fascinating topic).
The theorem says that a bounded function on the hyperbolic plane is uniquely determined by its integrals over all ideal triangles. In Section \ref{section:pompeiu} we prove, and discuss more thoroughly, Theorem \ref{thm:pompeiu_0}, rephrasing it as the injectivity of an \emph{ideal triangle transform}. As this is all we need for our present application, we do not investigate this transform further, even though it could be interesting to understand its range, to find explicit inversion formulas etc.~(see, e.g., \cite{volchkov_book} for integral geometric questions on symmetric spaces in a similar spirit as Theorem \ref{thm:pompeiu_0}). Our proof of Theorem \ref{thm:pompeiu_0} uses the hyperbolic version of a fundamental result in (Euclidean) Fourier analysis, the Wiener's Tauberian Theorem, established around 20 years ago by Mohanty, Ray, Sarkar, and Sitaram. See Section \ref{section:pompeiu} for details. We remark that this section of the paper has been written for a reader not previously acquainted with Fourier analysis on hyperbolic spaces. 

\bigskip

Theorem~\ref{thm:pompeiu_0} applies to $L^\infty$ functions which do not need to be smooth, and in fact our Theorem~\ref{main:thm} can be easily generalized, in the case of surfaces, to deal with differential forms with $L^\infty$ coefficients. However, if one allows differential forms to have non-smooth coefficients defined almost everywhere, then in higher dimensions even defining the map $\psi$ is not trivial at all. We will take care of this issue in a subsequent paper. 

\bigskip

One may wonder whether our Theorem~\ref{main:thm} might be generalized to deal with degree-$d$ smooth differential forms, $d\geq 3$. 
The hyperbolic volume of straight
simplices is uniformly bounded in every dimension not smaller than 2, hence the map $\psi\colon Z\Omega^d_b(M)\to H^d_b(M)$ is still well defined. However, if
$d\geq 3$ and a form $\omega\in Z\Omega^d_b(M)$ admits a bounded primitive $\alpha\in \Omega_b^{d-1}(M)$, then by integrating $\alpha$ on straight $(d-1)$-simplices
(which still have uniformly bounded volume, since $d\geq 3$) one obtains a bounded primitive of the singular cocycle associated to $\omega$, whence $\psi(\omega)=0$.
This shows that $\psi$ can never be injective: if $\alpha$ is any non-closed compactly supported $(d-1)$-form on $M$ (such a form always exists if $d\leq \dim M$), then $\omega=d\alpha$ is a non-trivial element
of $\ker \psi$.

However, it is maybe worth mentioning that, if $M$ is non-compact, there can be plenty of elements of $Z\Omega^d_b(M)$ which do not admit bounded primitives and which define
interesting bounded cohomology classes: for example, volume forms of geometrically infinite hyperbolic $3$-manifolds homotopy equivalent to a hyperbolic surface $S$ were exploited  in~\cite{soma1,soma2} to define
infinitely many linearly independent classes in
$H^3_b(S)$; in higher dimensions, the question whether the volume form of a manifold $M$ of infinite volume defines a non-trivial bounded cohomology class 
(i.e.~it does not belong to $\ker \psi$) has been investigated in~\cite{kim2015bounded,ervin_cheeger}.



\subsection{Plan of the paper} In Section~\ref{sec:preliminari} we introduce the bounded cohomology of topological spaces and we provide  the details
for the definition of the map $\psi\colon Z\Omega^2_b(M)\to H^2_b(M)$. Section~\ref{noninj:sec}  is devoted to the non-injectivity of $\psi$ for manifolds of the second kind.
In Section~\ref{boundary:sec} we exploit Burger and Monod's approach to bounded cohomology of groups via amenable actions
to translate the injectivity of $\psi$ into a problem of integral geometry on the hyperbolic plane. We solve this problem in Section~\ref{section:pompeiu}, which is devoted to the proof
of Theorem~\ref{thm:pompeiu_0}. Finally, in the last section we conclude the proof of our main result, Theorem~\ref{main:thm}.

\subsection{Acknowledgements} The authors thank Gabriele Viaggi and Fulvio Ricci for useful conversations. Roberto Frigerio and Ervin Had\v{z}iosmanovi\'{c} are partially supported by
 INdAM through GNSAGA. Gian Maria Dall'Ara is supported by INdAM.

\section{Preliminaries}\label{sec:preliminari}

\subsection{Bounded cohomology of spaces}
Let $X$ be a topological space. The \emph{bounded cohomology} $H^\bullet_b(X)$ of $X$ (with real coefficients) is the cohomology of the complex
\[
	0\rightarrow C^0_b(X)\rightarrow C^1_b(X)\rightarrow C^2_b(X)\rightarrow\cdots \,,
\]
where $C^n_{b}(X)$ denotes the space of bounded singular cochains of $X$, the differential maps are the restrictions of the boundary maps of the usual singular cochain complex $C^\bullet(X)$ and a cochain $\varphi \in C^n(X)$ is called \emph{bounded} if
\[
	\lVert \varphi \rVert_\infty =
	\sup \bigl\{ | c(\sigma) |, \; \sigma \mbox{ is a singular $n$-simplex}
	\bigr\} < \infty .
\]

\subsection{Bounded cohomology and differential forms}\label{sub:derham clas}
Let $M$ be a complete connected hyperbolic $n$-manifold. Then $M$ is isometric to $\Hn/\Gamma$, where $\Gamma\cong \pi_1(\Sigma)$
is a torsion-free Kleinian group, i.e.~a torsion-free discrete subgroup of $\isom(\Hn)$. The hyperbolic metric on $M$ allows us to define a \emph{straightening} procedure for simplices as follows:
if $\tilde{s}$ is a singular simplex of $\Hn$, its straightening $\strtil(\tilde{s})$ is a suitable smooth parametrization of the geodesic singular simplex in $\Hn$ having the same vertices as $\tilde{s}$;
if $s$ is a singular simplex with values in $M$, then its straightening  $\str(s)$ is the image in $M=\Hn/\Gamma$ of $\strtil(\tilde{s})$, where 
$\tilde{s}$ is any lift of $s$ to $\Hn$. The singular simplices $\strtil(\tilde{s})$ and $\str(s)$ are called \emph{straight}. In fact, the singular simplex $\strtil(\tilde{s})$
only depends on its vertices, hence in the sequel we will speak of \emph{the} straight simplex in $\Hn$ with a given ordered set of vertices. Moreover,
the linear extension of the straightening operator to the space of singular chains defines a chain map
$\str\colon C_\bullet (M)\to C_\bullet (M)$, which is chain homotopic to the identity
(for the details see, e.g., \cite[Section 8.4]{Frigeriobook}).

Let $\Omega^2(M)$ be the space of smooth differential $2$-forms on $M$.
For every $\omega\in\Omega^2(M)$ we set
$$
\|\omega\|_\infty=\sup \{|\omega_p(e_1,e_2)|\, ,\ p\in M\, ,\ (e_1,e_2)\ \text{orthonormal}\ 2-\text{frame in}\ T_p M\}\ .
$$
We say that $\omega$ is \emph{bounded} if $\|\omega\|_\infty<\infty$ (this condition always holds if $M$ is compact), and we denote by
$\Omega^2_b(M)\subseteq \Omega^2(M)$ the subspace of bounded 2-forms. 

Every $\omega \in \Omega^2(M)$ defines a singular $2$-cochain $c_\omega \in C^2(M)$ in the following way: for every singular $2$-simplex $s$, 
$$c_\omega(s)=\int_{\str(s)} \omega\ .$$ 
Since $\str(s)$ is smooth, the cochain $c_\omega$ is indeed well defined; moreover, $c_\omega(s)$  is the integral of 
$\widetilde{\omega}$ on the straightening of any lift of $s$, where $\widetilde{\omega}$ denotes the pull-back of $\omega$ to the universal covering $\mathbb{H}^n$.
Since the area of geodesic triangles in $\Hn$ is bounded from above by $\pi$, we have
$\|c_\omega\|\leq \pi \|\omega\|_\infty$, hence the singular cochain $c_\omega$ is bounded 
whenever
$\omega$ is bounded. Finally, by Stokes' Theorem (together with the fact that the straightening is a chain map), if $\omega$ is closed
then $c_\omega$ is a cocycle. Therefore, if we denote by $Z\Omega^2_b(M)$ the space of bounded closed differential $2$-forms on $M$,
then we have a well-defined map
$$
\psi\colon Z\Omega^2_b(M)\to H^2_b(M)\, ,\qquad \psi(\omega)=[c_\omega]\ .
$$

With this notation,
Barge and Ghys' Theorem (and its extension by Battista et al. to higher dimensions) may be restated as follows:

\begin{thm}[{\cite[Theorem 3.2]{bargeghys} and \cite[Theorem 1]{Battista}}]\label{BG:thm}
Suppose that $M$ is compact.
Then the map $\psi\colon Z\Omega^2_b(M)\to H^2_b(M)$ is injective.
\end{thm}

\section{Injectivity fails for Kleinian groups of the second kind}\label{noninj:sec}

We first prove that, if $M=\Hn/\Gamma$ is a hyperbolic manifold of the second kind, then $\psi$ is not injective. Let us recall that a subset $C\subseteq M$ is \emph{convex} if the following condition holds: if $p,q\in C$, then \emph{every} geodesic in $M$ joining $p$ to $q$ is contained in $C$
(with this definition, if $\Gamma\neq \{1\}$, then a single point is not convex). The \emph{convex core} of $M$ is the smallest closed convex subset $C$ of $M$. If $\Gamma$
is non-elementary, then it is well-known that $C$ is the projection in $M$ of the convex hull of the limit set $\Lambda(\Gamma)$ of $\Gamma$. In this case, moreover,
there is a strong deformation retraction of $M$ onto $C$, i.e.~a homotopy $H\colon M \times [0,1]\to M$ such
that $H(\cdot, 0)=\textrm{Id}_M$, $H(x,t)=x$ for every $x\in C$ and $H(x,1)\in C$ for every $x\in M$. 

\begin{prop}\label{convexcore}
	Suppose that $\Gamma$ is not elementary, let $C$ be the convex core of $M$, and let $\text{res}\colon Z\Omega^2_b (M)\to \Omega^2_b (C)$ be the restriction map. 
	Then
	$$
	\ker \mathrm{res} \subseteq \ker \psi\ .
	$$ 
\end{prop}
\begin{proof}
	Let $r\colon M\to C$ be the retraction described above. Since $r$ is homotopic to the identity and homotopic maps induce the same map in bounded cohomology,
	for every $\omega\in Z\Omega^2_b(M)$ we have $r^*(\psi(\omega))=\psi(\omega)$. Let now $s\colon \Delta^2 \to M$ be a singular $2$-simplex. The simplex $r_*(s)$ has all its vertices
	in $C$, hence, by convexity of $C$, its straightening $\str(r_*(s))$ is contained in $C$. Therefore, if $\omega\in \ker \textrm{res}$, i.e.~$\omega|_C=0$, then
	$$
	r^*(c_\omega)(s)=c_\omega (r_*(s))=\int_{\str(r_*(s))} \omega=0\ .
	$$
	Therefore $r^*(\psi(\omega))$ is represented by the trivial cocycle, and $\psi(\omega)=r^*(\psi(\omega))=0$. This concludes the proof.
	\end{proof}

\begin{cor}\label{noninjectivity}
If $M$ is of the second kind, then the map $\psi\colon Z\Omega^2_b(M)\to H^2_b(M)$ is not injective.
\end{cor}
\begin{proof}
Let $M=\Hn/\Gamma$.
If $\Gamma$ is elementary,
then it is amenable, and $H^2_b(\Gamma)=0$, so $\psi$ is not injective; otherwise, since $M$ is of the second kind, the complement 
$U=M\setminus C$ of the convex core
 is non-empty and open,
and we can easily construct a non-trivial element of $Z\Omega^2_b(M)$ supported in $U$ (just take, for example, the differential
of a bump differential $1$-form  with compact support contained in $U$). Thus
the restriction map $\text{res}\colon Z\Omega^2_b (M)\to \Omega^2_b (C)$ has a non-trivial kernel, and the conclusion follows from Proposition~\ref{convexcore}.
\end{proof}

\section{Boundary representations of bounded classes}\label{boundary:sec}
The bounded class $\psi(\omega)$ corresponding to a bounded  differential $2$-form $\omega\in Z\Omega^2_b(M)$ 
can also be described by integration on ideal triangles with vertices on $\partial \Hn$, rather than on straight triangles with vertices inside $\Hn$.
In order to make this statement precise, one needs to switch from the bounded cohomology of $M$ to the bounded cohomology of its fundamental group.

\subsection{Bounded cohomology of groups}\label{group:sub}
Let $G$  be a discrete group.
The \emph{bound\-ed cohomology} of $G$ (with real coefficients), denoted by $H^\bullet_b(G)$, is defined as the cohomology of the following complex of vector spaces
\[
	0 \rightarrow C^0_{b}(G)^G \rightarrow C^1_{b}(G)^G\rightarrow C^2_{b}(G)^G \rightarrow \cdots\,,
\]
where $C^n_{b}(G)^G$ denotes the space of bounded $G$-invariant maps from $G^{n+1}$ to $\R$, the differential maps are defined by $$(\delta \varphi)(\gamma_0,\dots,\gamma_n) = \sum_{i=0}^n(-1)^i\varphi(\dots,\widehat{\gamma_i},\dots)\,,$$
and the action of $G$ on $C^n_b(G)$ is given by $(\gamma\cdot \varphi)(\gamma_0,\dots,\gamma_n)=\varphi(\gamma^{-1}\gamma_0,\dots, \gamma^{-1}\gamma_n)$.

A celebrated theorem due to Gromov ensures that the bounded cohomology of a topological space is isometrically isomorphic to the bounded cohomology of its fundamental group.
In fact, when the topological space supports a straightening procedure, Gromov's isomorphism admits a very explicit description, which we recall
here (in degree 2) in the case of hyperbolic manifolds. 
Let $M=\Hn/\Gamma$, let ${p}\in\Hn$ be a fixed basepoint, and 
for every $(\gamma_0,\gamma_1,\gamma_2)\in\Gamma^3$
let $T(\gamma_0 p,\gamma_1 p,\gamma_2 p)$ denote the straight triangle of $\Hn$ with vertices $\gamma_0 p,\gamma_1 p,\gamma_2 p$ (recall that straight simplices
are singular simplices, hence $T(\gamma_0 p,\gamma_1 p,\gamma_2 p)$ is a parametrizazion of the geodesic triangle with vertices  $\gamma_0 p,\gamma_1 p,\gamma_2 p$, i.e.~of the hyperbolic
convex hull of the points $\gamma_0 p,\gamma_1 p,\gamma_2 p$).
We then consider the map $\theta \colon C^2_b(M) \rightarrow C^2_b(\Gamma)^\Gamma$ defined by 
\[(\theta(\varphi))(\gamma_0,\gamma_1,\gamma_2)=\tilde{\varphi}(T(\gamma_0 p,\gamma_1 p,\gamma_2 p)),\]
where, for every singular cochain $\varphi\in C^2_b(M)$, we denote by
$\tilde{\varphi} \in C^2_b(\Hn)^\Gamma$ the $\Gamma$-invariant lift of $\varphi$.
It is well known that $\theta$ induces an isomorphism in bounded cohomology (see e.g.~\cite[Corollary 4.15]{Frigeriobook}), which will still be denoted by $\theta$, with
a slight abuse. By construction,
if $\omega\in Z\Omega^2_b(M)$ and $\widetilde{\omega}\in Z\Omega^2_b(\Hn)$ is the pull-back of $\omega$ to $\Hn$, then $\theta(\psi(\omega))$ is represented by the cocycle
$$
c_{\omega,{p}}\in C^2_b(\Gamma)^\Gamma\, ,\qquad c_{\omega,{p}}(\gamma_0,\gamma_1,\gamma_2)=\int_{T(\gamma_0 {p}, \gamma_1 {p}, \gamma_2 {p})}\tilde \omega\ .
$$

A key step in our approach to the study of $\psi$ consists in representing the class $\theta(\psi(\omega))$ ``on the boundary'' of $\Hn$.
To this aim, we denote by $\calT$ the set of elements $(\xi_0,\xi_1,\xi_2)\in (\bi\Hn)^3$ such that $\xi_i\neq \xi_j$ for $i\neq j$. 
Thus, $\mathcal{T}$ parametrizes the set of ideal triangles
with ordered vertices in $\mathbb{H}^n$. For every $(\xi_0,\xi_1,\xi_2)\in \mathcal{T}$, we denote by  of $T(\xi_0,\xi_1,\xi_2)$ an
\emph{oriented} parametrization  of the ideal triangle $T$ with vertices $\xi_0,\xi_1,\xi_2$, i.e.~the restriction to $\Delta^2\setminus\{v_0,v_1,v_2\}$ of a parametrization $\sigma\colon \Delta^2 \to \overline{T}\subseteq \overline{\mathbb{H}^n}$ 
sending $v_i$ to $\xi_i$ for $i=0,1,2$, where $v_0,v_1,v_2$ are the vertices of $\Delta^2$.

Let now $\xi\in\bi\Hn$ be a fixed basepoint. We consider the cochain
$c_{\omega,\xi}\in C^2_b(\Gamma)^\Gamma$ defined  by
$$
c_{\omega,\xi}(\gamma_0,\gamma_1,\gamma_2)=\left\{\begin{array}{ll} \int_{T(\gamma_0 {\xi}, \gamma_1 {\xi}, \gamma_2 {\xi})}\tilde \omega \quad & \text{if}\ 
(\gamma_0 {\xi}, \gamma_1 {\xi}, \gamma_2 {\xi})\in\mathcal{T}\\
0 \quad &\text{otherwise}\ .\end{array}\right.
$$

The following result was proved by Barge and Ghys when  $n=2$, but their proof
applies \emph{verbatim} also to higher dimensions.

\begin{lemma}[{\cite[Lemma 3.10]{bargeghys}}]\label{boundaryrep:lemma}
For every $\xi\in\bi\Hn$, the cochain $c_{\omega,\xi}$ is a bounded cocycle representing
the class $\theta(\psi(\omega))\in H^2_b(\Gamma)$.
\end{lemma}
\qed

\subsection{Bounded cocycles as measurable functions}\label{sec: bc as meas fun}

We are now going to exploit some fundamental results concerning the use of amenable actions
for the computation of bounded cohomology.

\begin{defn}
Let $G$ be a group. Following
\cite[Definiton 2.1.1]{Monod}, we say that a \emph{regular $G$-space} is a standard Borel space $S$ on which $G$ acts measurably, together with a $G$-invariant measure
 class with the following property: the measure class contains a probability measure $\mu$ turning $(S,\mu)$ into a standard probability space such that the natural isometric
 $G$ action on $L^1(\mu)$ defined by
  $$
 (g\cdot f)(s)=f(g^{-1}\cdot s)\frac{dg^{-1}\mu}{d\mu}(s)\, ,\qquad f\in L^1(\mu)\, ,\ s\in S
 $$
 is continuous.
 \end{defn}

Recall that a discrete group $G$ is  \textit{amenable} if it admits a left-invariant mean on the space $\ell^\infty(G)$ of real-valued bounded functions
on $G$.
Basic examples of amenable groups are finite groups, abelian groups and extensions of amenable groups by amenable groups (which include solvable and nilpotent groups). A classic result of Trauber is that amenable groups have vanishing bounded cohomology.
The definition of amenability for discrete groups can be extended to the locally compact case and, more in general, to $G$-spaces (see \cite[Chapter 4]{zimmer} and \cite[Chapter II, Definition 5.3.1 and Theorem 5.3.2]{Monod}). If $G$ is itself amenable, then any regular $G$-space  also is, but non-amenable groups may well act amenably on regular spaces.
The fundamental r\^ole of amenable actions in the theory of bounded cohomology is described by Theorem~\ref{amenable_resolution} below, which states that,
if  $S$ is an amenable regular $G$-space, then the bounded cohomology of $G$ is computed by the cochain complex 
\begin{equation}\label{resolution}
	0\rightarrow L^\infty_{\rm alt}(S)^G\rightarrow  L^\infty_{\rm alt}(S^2)^G\rightarrow L^\infty_{\rm alt}(S^3)^G\rightarrow\cdots 
\end{equation}
of $G$-invariant alternating measurable bounded functions on $S^{n+1}$ up to equality almost everywhere, endowed with the usual differential

\[(\delta \varphi)(x_0,\dots,x_n) = \sum_{i=0}^n(-1)^i \varphi(\dots,\widehat{x_i},\dots)\ .\]
Recall that a cochain $\varphi\colon S^{n+1}\to\mathbb{R}$ is \emph{alternating} if $$\varphi(x_{\tau(0)},\dots,x_{\tau(n)})=\varepsilon(\sigma)\varphi(x_0,\dots,x_n)$$ for every
permutation $\tau$ of $\{0,\dots,n\}$, where $\varepsilon(\tau)$ is the sign of $\tau$.

\begin{thm}[{\cite[Theorem 2]{burger2001continuous}\cite[Theorem 7.5.3]{Monod}}]\label{amenable_resolution}
	Let $G$ be a discrete group and $S$ be an amenable regular $G$-space. Then, the cohomology of the complex~\eqref{resolution}
	is canonically isomorphic to $H^\bullet_b(G)$. 
\end{thm}

\begin{rem}
The above theorem works in the more general framework of \textit{continuous bounded cohomology} of a locally compact, second countable group and for arbitrary coefficient modules. In what follows, only bounded cohomology of discrete groups with real coefficients will be considered and thus it is sufficient to state the result in this less general form.
\end{rem}

	The elements of $L^\infty_{\rm alt}\left(S^{n+1} \right)$ are not functions, but \emph{classes} of functions.
	This being said, as it is customary we will usually denote the class of an element in $ L^\infty_{\rm alt}\left(S^{n+1} \right)^G$ simply by one of its representatives, thus writing  $ c\in L^\infty_{\rm alt}\left(S^{n+1} \right)^G$
	also when $c$ is a function.
	Notice that such a $c$ is not necessarily $G$-invariant as a function, even if its class is, and if $c$ is a cocycle, then the cocycle condition $\delta c = 0$ only holds almost everywhere in general. For later purposes, we need to single out the cases in which we can work with functions rather than with classes of functions. To this aim, we give the following:

\begin{defn}
We denote by $\mathcal{L}^\infty_{\text{alt}}(S^{\bullet+1})^G$  
 the set of bounded measurable alternating functions 
$S^{\bullet+1}\to\mathbb{R}$ which are genuinely (i.e.~not only almost everywhere) $G$-invariant. 

A \emph{strict} cocycle is an element $c\in \mathcal{L}^\infty_{\text{alt}}(S^{\bullet+1})^G$ such that
the identity $\delta c=0$ holds everywhere on $S^{\bullet+2}$ (in this case we write $c\in Z\mathcal{L}^\infty_{\text{alt}}(S^{\bullet+1})^G$). 
Any strict cocycle $c\in Z\mathcal{L}^\infty_{\text{alt}}(S^{n+1})^G$ defines a
cocycle in $L^\infty_{\rm alt}(S^{n+1})^G$, hence, if $S$ is an amenable regular $G$-space, it represents an element of $H_b^n(G)$.
\end{defn}

Even when $S$ is an amenable regular $G$-space, it is not clear whether the complex $\mathcal{L}^\infty_{\text{alt}}(S^{\bullet+1})^G$ computes the bounded cohomology of $G$:
as observed in~\cite[Section 2.7]{BurgerIozzi2}, the resolution $\mathcal{L}^\infty_{\text{alt}}(S^{\bullet+1})$ is strong, but we don't know whether the $G$-modules 
$\mathcal{L}^\infty_{\text{alt}}(S^{\bullet+1})$ are relatively injective. However, strict cocycles may prove very useful when looking for explicit representatives of classes in $H^\bullet_b(\Gamma)$
(see again~\cite{BurgerIozzi} and~\cite{BurgerIozzi2} for a thorough discussion of this issue).

We will be mainly interested in the case when $S=\partial\mathbb{H}^n$
and $G=\Gamma$ is a discrete subgroup of $\text{Isom}^+(\Hn)$. 
We endow $\bi\Hn$ with the measure class $\mu_L$ of the Liouville measure
(which coincides with the measure class of the Lebesgue measure on $S^{n-1}=\partial D^n$, after identifying
$\Hn$ with the Poincar\'e disc model $D^n$, and with the measure class of the  visual measure associated to any basepoint in $\Hn$).

It is well known (and very easy to check using e.g.~that every element of $\Gamma$ acts as a diffeomorphism of $\partial \Hn$) 
that $(\partial\Hn,\mu_L)$ is indeed a regular $\Gamma$-space.
In fact, it is also amenable:

\begin{lemma}\label{amenability}
Let $\Gamma$ be a discrete subgroup of $\operatorname{Isom}^+(\Hn)$. Then $(\partial\Hn,\mu_L)$ is an amenable regular $\Gamma$-space.
\end{lemma}
\begin{proof}
The group $G=\text{Isom}^+(\Hn)$ is a semisimple Lie group.
Since the action of $G$ on $\partial \Hn$ is transitive, if 
$P=\operatorname{Stab}(x_0)<G$ for some $x_0\in\partial \Hn$, then 
we have an identification 
$\partial \Hn\cong G/P$ as $G$-spaces. Moreover, $G/P$ admits a unique invariant measure class \cite[Example 2.1.2]{zimmer}, which, therefore, coincides with $\mu_L$ under the identification $\partial \Hn\cong G/P$. Since $P$ is amenable,  \cite[Proposition 4.3.2]{zimmer} then implies that $\partial \Hn$ is an amenable $G$-space. Finally, since $\Gamma$ is discrete and thus closed
in $G$, the conclusion follows from the fact  that restrictions of amenable actions to closed subgroups are amenable
(see~\cite[Proposition 4.5.3]{zimmer} or~\cite[Lemma 5.4.3]{Monod}).
\end{proof}

Putting together Theorem~\ref{amenable_resolution} and Lemma~\ref{amenability} we obtain the following:
\begin{cor}\label{realization:cor}
The bounded cohomology of $\Gamma$ is isometrically isomorphic to the cohomology of the complex
\begin{equation}\label{resolution2}
	0\rightarrow L^\infty_{\rm alt}(\partial \Hn)^\Gamma\rightarrow  L^\infty_{\rm alt}((\partial \Hn)^2)^\Gamma\rightarrow L^\infty_{\rm alt}((\partial \Hn)^3)^\Gamma\rightarrow\cdots 
\end{equation}
\end{cor}

Let us now focus on the situation we are interested in, i.e.~let us take ${\omega}\in Z\Omega^2_b(M)$, where $M=\Hn/\Gamma$, and
let $\widetilde{\omega}\in Z\Omega^2_b (\Hn)$ be the pull-back of
$\omega$ on $\Hn$. Recall that
$\calT$ denotes the set of elements $(\xi_0,\xi_1,\xi_2)\in (\bi\Hn)^3$ such that $\xi_i\neq \xi_j$ for $i\neq j$. 
We consider the 
map $c_\omega\colon (\bi\Hn)^3\to \mathbb{R}$
given by
$$
c_{\omega}(\xi_0,\xi_1,\xi_2)=\left\{\begin{array}{ll} \int_{T(\xi_0,\xi_1,\xi_2)} \widetilde{\omega}\qquad &\text{if}\ (\xi_0,\xi_1,\xi_2)\in\mathcal{T}\ ,\\
0 \qquad & \text{otherwise.}\end{array}\right.
$$

\begin{prop}\label{properties:c}
We have $c_\omega\in Z\mathcal{L}^\infty_{\text{alt}}((\bi\Hn)^{3})^\Gamma$, i.e.~$c_\omega$ is a strict cocycle. Moreover, its restriction to $\mathcal{T}$ is continuous.
\end{prop}
\begin{proof}
Since the area of ideal triangles is equal to $\pi$, we have 
$\|c_\omega\|_\infty\leq \pi \|\omega\|_\infty$, hence $c_\omega$ is well defined (as a set-theoretic function) and bounded; moreover, the $\Gamma$-invariance of $\widetilde{\omega}$  implies the $\Gamma$-invariance of $c_\omega$, and $c_\omega$ is alternating by construction.  The cocycle identity
$\delta c_\omega=0$ readily follows from the fact that $\widetilde{\omega}$ is closed, together with Stokes' Theorem, hence
we are left to show that  $c_\omega$ is continuous on $\mathcal{T}$ (which implies in turn that $c_\omega$ is  measurable
on $(\partial \mathbb{H}^n)^3$).

Let $(\overline{\xi}_0, \overline{\xi}_1,\overline{\xi}_2)$ be an element of $\mathcal{T}$, and let
$({\xi}_0^n, {\xi}_1^n,{\xi}_2^n)$ be a sequence in $\mathcal{T}$ converging to $(\overline{\xi}_0, \overline{\xi}_1,\overline{\xi}_2)$. 
Using that $\overline{\xi}_i\neq \overline{\xi}_j$ for $i\neq j$, it is easy to prove that we may choose parametrizations
of $T({\xi}_0^n, {\xi}_1^n,{\xi}_2^n)$ which converge to a parametrization of $T(\overline{\xi}_0, \overline{\xi}_1,\overline{\xi}_2)$ in the $C^1$-topology. 
Since $\widetilde{\omega}$ is bounded, we may then apply the Dominated Convergence Theorem to show that
$$
\lim_{n\to \infty} c_\omega ({\xi}_0^n, {\xi}_1^n,{\xi}_2^n)=\lim_{n\to \infty} \int_{T({\xi}_0^n, {\xi}_1^n,{\xi}_2^n)} \widetilde{\omega}
=\int_{T(\overline{\xi}_0, \overline{\xi}_1,\overline{\xi}_2)} \widetilde{\omega}=c_\omega(\overline{\xi}_0, \overline{\xi}_1,\overline{\xi}_2)\ .
$$
This shows that $c_\omega$ is continuous on $\mathcal{T}$ and conclude the proof.
\end{proof}

\begin{prop}\label{rep:boundary}
The strict cocycle $c_\omega$ represents the class $\theta(\psi(\omega))\in H^2_b(\Gamma)$.
\end{prop}

\begin{proof}
Recall from \cite[Theorem 2]{burger2001continuous} or \cite[Theorem 7.5.3]{Monod}
that the canonical isomorphism between the bounded cohomology of $\Gamma$ and the cohomology of the complex $L^\infty_{\rm alt}((\partial \Hn)^{\bullet +1})^\Gamma$
is induced by any $\Gamma$-equivariant 
chain map $\alpha^\bullet\colon L^\infty_{\rm alt}((\partial \Hn)^{\bullet +1})\to C_b^{\bullet+1}(\Gamma)$ such that the following diagram commutes: 
$$
\begin{CD}
\mathbb{R} @>\varepsilon>> L_{\rm alt}^\infty(\partial \Hn)^\Gamma\\
@VV{\rm{Id}}V @VV\alpha^0 V\\
\mathbb{R} @>\varepsilon'>> C_b^{0}(\Gamma)^\Gamma
\end{CD}
$$
where $\varepsilon$ and $\varepsilon'$ denote the usual augmentation maps.	

If we could  define $\alpha^{k}\colon L_{\rm alt}^\infty((\partial \mathbb{H}^n)^{k+1})^\Gamma\to C^k_b(\Gamma)^\Gamma$ by choosing a point $\xi\in\partial \Hn$ and setting 
$\alpha^k(\varphi)(\gamma_0,\dots,\gamma_k)= \varphi (\gamma_0\xi,\dots,\gamma_k\xi)$ for every $\varphi\in  L_{\rm alt}^\infty((\partial \mathbb{H}^n)^{k+1})^\Gamma$, then 
the conclusion would immediately follow from Lemma~\ref{boundaryrep:lemma}.
Unfortunately, since 
elements in $L_{\rm alt}^\infty((\partial \mathbb{H}^2)^{k+1})$ are defined only almost everywhere,  the chain map $\alpha^{k}$ just described is
not  well defined. However, is it now of help the fact that $c_\omega$ is a strict cocycle, i.e.~that it is defined everywhere and is such that
$\delta c_\omega=0$ everywhere. Thanks to this additional property of $c_\omega$, one can make the above strategy work; indeed, by
applying~\cite[Corollary 2.3]{BurgerIozzi} to the case when  $G=\Gamma$,  $X=\bi\Hn$, and $Z=\Gamma\cdot \xi\subseteq X$, where $\xi\in\bi\Hn$ is a fixed basepoint, we can conclude
that the cochain $c_\omega$ represents the class $[c_{\omega,\xi}]\in H^2_b(\Gamma)$, where $c_{\omega,\xi}\in C^2_b(\Gamma)^\Gamma$ is the cocycle described in Subsection~\ref{group:sub}.
The conclusion then follows from Lemma~\ref{boundaryrep:lemma}.

\end{proof}

We are now ready to translate the question whether $\psi$ is injective into a problem of integral geometry on the hyperbolic space.
Let us first describe how this can be achieved when the geodesic flow on $M$ is ergodic (it is well known that this condition is stronger than
asking $M$ to be of the first kind,  see Remark~\ref{dynamic:rem} below). In this case, the action of $\Gamma$ on $\partial \Hn$ is doubly ergodic. This readily implies that any measurable alternating
$\Gamma$-invariant function on $\bi\Hn\times \bi\Hn$ vanishes almost everywhere, hence the space $L^\infty_{\rm alt} ((\bi\Hn)^3)^\Gamma$
contains no non-trivial coboundaries. Therefore, a bounded cocycle in $L^\infty_{\rm alt} ((\bi\Hn)^3)^\Gamma$ represents the trivial bounded cohomology
class if and only if it vanishes almost everywhere. In our case of interest, this implies that $\omega$ belongs to $\ker \psi$ if and only if its 
representative $c_{\omega}\in \mathcal{L}^\infty_{\text{alt}}((\bi\Hn)^3)^\Gamma$ vanishes almost everywhere, hence everywhere (recall that
$c_\omega$ is continuous on $\mathcal{T}$ and vanishes outside $\mathcal{T}$). In other words, $\omega$ belongs to $\ker\psi$ if and only if 
$$
\int_{T(\xi_0,\xi_1,\xi_2)} \widetilde{\omega}=0
$$
for every ideal triangle $T(\xi_0,\xi_1,\xi_2)\in\mathcal{T}$. 

In fact, building on the mere existence of a suitable measure space on which $\Gamma$ acts amenably and doubly ergodically (in the literature, such a space is often called a \emph{strong boundary}
for $\Gamma$), a clever argument described in~\cite{BurgerIozzi} allows us to prove that the same conclusion holds under the weaker hypothesis that $M$ be of the first kind:

\begin{thm}\label{annullamento:thm}
Suppose that $M=\mathbb{H}^n/\Gamma$ is a hyperbolic manifold of the first kind, and take $\omega\in Z\Omega^2_b(M)$. Then
$\psi(\omega)=0$ if and only if 
$$
\int_{T(\xi_0,\xi_1,\xi_2)} \widetilde{\omega}=0
$$
for every  $(\xi_0,\xi_1,\xi_2)\in\mathcal{T}$.
\end{thm}
\begin{proof}
By Proposition~\ref{rep:boundary}, 
the statement of the theorem is equivalent to the request that
$c_\omega$ represents the trivial class in bounded cohomology if and only if it vanishes identically.
The ``if'' implication is obvious, while for the ``only if'' implication
we exploit a result from~\cite[Section 3]{BurgerIozzi}; to this aim, we first introduce the notation
used therein. 

Let $X$ be a proper CAT$(-1)$-space, $G_2$ a closed subgroup of $\isom(X)$, $E$ a separable 
coefficient $G_2$-module (i.e.~the topological dual of a Banach $G_2$-module), and
$c\colon \partial X^3\to E$ 
a Borel measurable, alternating, bounded, $G_2$-invariant strict cocycle which is continuous
on the subset $\mathcal{T}$ of pairwise distinct triples in $(\partial X)^3$. 
Also denote by $[c]\in H^2_b(G_2)$ the class represented by $c$ via the Burger-Monod isomorphism
between the bounded cohomology of $G_2$ and the cohomology of the resolution
$L^\infty_{\rm alt}((\partial X)^{\bullet +1})^{G_2}$. 
Finally, let
$\pi\colon G_1\to G_2$ be 
a homomorphism, where $G_1$ is discrete, and denote by $\Lambda(\pi(G_1))\subseteq \partial X$ the limit
set of $\pi(G_1)$.

It is then proved in~\cite[Proposition 3.1]{BurgerIozzi} that if the pull-back $\pi^*([c])\in H^2_b(G_1)$ 
vanishes, then the restriction of $c$ to $\Lambda(\pi(G_1))^3\subseteq (\partial X)^3$ is identically zero. 
 We apply this result to the case when $X=\Hn$, $E=\mathbb{R}$ on which $G_2$ acts trivially, $G_1=G_2=\Gamma$, $\pi\colon \Gamma\to\Gamma$ is the identity, and $c=c_\omega$, thus obtaining that $c_\omega$ represents the trivial class of $H^2_b(\Gamma)$
 only if it vanishes identically. This concludes the proof.
\end{proof}
	
\begin{rem}\label{dynamic:rem}
In general, for a complete hyperbolic manifold $M=\Hn/\Gamma$, the relation between the ergodicity of the geodesic flow and the fact that $M$ is of the first kind is quite subtle. 
The classical Hopf-Tsuji-Sullivan Theorem (see~\cite[Theorem 8.3.4]{nicholls1989ergodic}) ensures that 
the geodesic flow of $M$ is ergodic if and only if the  action of $\Gamma$ on $\partial \mathbb{H}^{n}$ is doubly ergodic.
This condition implies in particular that the action of $\Gamma$ on 
$\bi\Hn\times \bi\Hn$, whence, \emph{a fortiori}, on $\bi\Hn$, is topologically transitive, and this implies in turn that $\Gamma$ is of the first kind.

Conversely, as mentioned in the introduction, it is not true that the geodesic flow of every hyperbolic manifold of the first kind is ergodic. 
However, the ergodicity of the geodesic flow and the condition of being of the first kind
are equivalent in dimension 2 and 3, under the assumption that $\Gamma$ is finitely generated.
In fact, for surfaces,
if $\Gamma$ is finitely generated, then the area of $M$ is finite if and only if $M$ is of the first kind (see e.g.~\cite[Theorems 10.1.2 and 10.4.3]{beardon});
moreover,  the geodesic flow of finite area hyperbolic surfaces is ergodic
\cite{Hopf}. In dimension 3, the equivalence between the two conditions follows from 
the solution of Marden's Tameness Conjecture~\cite{Ago,CalGab} together with the ergodicity of the geodesic flow of topologically tame hyperbolic $3$-manifolds
 of the first kind~\cite[Corollary 2]{Canary}. 
 
 In higher dimensions, to the best of the authors' knowledge, it is not known whether there  exist hyperbolic manifolds of the first kind with finitely generated fundamental group whose geodesic flow is not ergodic.
  \end{rem}

\section{Injectivity of the ideal triangle transform}\label{section:pompeiu}

In this section we prove Theorem \ref{thm:pompeiu_0}. Before going into the proof, we begin by placing this result in its natural context, namely that of \emph{integral geometry}. 

\subsection{Theorem \ref{thm:pompeiu_0} and integral geometry}\label{sec:preliminary_pompeiu} Integral geometry (see, e.g., \cite{gelfand_book}) investigates a variety of transforms associating to a function $f$ defined on some space $X$ its integrals 
\begin{equation}\label{eq:integrals}
	\int_S f\, d\mu_S,
\end{equation}
where $S$ varies in a family $Y$ of submanifolds (or more general subsets) of $X$ equipped with measures $\mu_S$. Typically, both the space $X$ and the family $Y$ carry geometric structures with a symmetry group $G$, and the transform commutes with the action of $G$. The prototypical example is the \emph{Radon transform} $R$ that integrates a function $f$ defined on the Euclidean plane $\R^2$ over all lines, thus yielding a function $Rf$ defined on the two-dimensional Grassmann manifold $\mathrm{Gr}_1(\R^2)$ of such lines. Here $X=\R^2$, $Y=\mathrm{Gr}_1(\R^2)$, and the symmetry group is the Euclidean motion group $M(2)$. 

A basic question in integral geometry is whether the transform of interest is injective, that is, whether a function $f$ is determined by the values of the integrals \eqref{eq:integrals}. The answer to this question depends not only on the transform, but also on the class of functions where $f$ is allowed to vary. The main tool used in these investigations is Fourier analysis on the symmetry group $G$. \newline 

Theorem \ref{thm:pompeiu_0} may be viewed as the solution of the above basic injectivity question for the \emph{ideal triangle transform}, which we define as the operator associating to a function $f$ on the hyperbolic plane $\mathbb{H}^2$ its integrals over all ideal triangles, i.e., the quantities \[
\int_{T(\xi_0,\xi_1,\xi_2)} f\, dA, \qquad (\xi_0,\xi_1,\xi_2)\in \mathcal{T}.
\]
As in the introduction, $dA$ is the hyperbolic area element, and we are using the notation introduced in Section \ref{group:sub} for ideal triangles. For our purposes, it is clearly sufficient to restrict consideration to the set $\mathcal{T}^+$ of positively oriented triples $(\xi_0,\xi_1,\xi_2)$ of distinct elements of the ideal boundary. Once such a triple $(\xi_0,\xi_1,\xi_2)$ is fixed, the set $\mathcal{T}^+$ may be identified with the group of orientation preserving isometries $\mathrm{Isom}^+(\mathbb{H}^2)$ via $g\mapsto (g\xi_0,g\xi_1, g\xi_2)$. 
Thus, we define the ideal triangle transform of a function $f\in L^\infty(\mathbb{H}^2)$ as \[
\mathcal{I}f(g) = \int_{T(g\xi_0,g\xi_1,g\xi_2)} f\, dA\qquad (g\in \mathrm{Isom}^+(\mathbb{H}^2)). 
\]
Since ideal triangles have area $\pi$, the boundedness assumption on $f$ guarantees the absolute convergence of the above integrals. We obtain a linear operator \[
\mathcal{I}:L^\infty(\mathbb{H}^2)\longrightarrow L^\infty(\mathrm{Isom}^+(\mathbb{H}^2))
\] 
satisfying the bound $\lVert\mathcal{I}f\rVert_{L^\infty(\mathrm{Isom}^+(\mathbb{H}^2))}\leq \pi \lVert f\rVert_{L^\infty(\mathbb{H}^2)}$. A standard dominated convergence argument shows that $\mathcal{I}f(g)$ depends continuously on $g$. Finally, it is clear that $\mathcal{I}$ commutes with the action of $\mathrm{Isom}^+(\mathbb{H}^2)$, that is, $\mathcal{I}f(h^{-1}g) = \mathcal{I}(f\circ h^{-1})(g)$ for all $g,h\in \mathrm{Isom}^+(\mathbb{H}^2)$. Theorem \ref{thm:pompeiu_0} can be rephrased as follows. 

\begin{thm}\label{thm:pompeiu}
	The ideal triangle transform $\mathcal{I}$ is injective on $L^\infty(\mathbb{H}^2)$.  	
\end{thm}

Theorem \ref{thm:pompeiu} may be viewed as a hyperbolic instance of the so-called Pompeiu's problem. A rapid discussion of the latter will help motivate the proof of Theorem \ref{thm:pompeiu}. 

The original Pompeiu's problem (in its two-dimensional version) concerns injectivity of transforms mapping a locally integrable function $f:\R^2\rightarrow \R$ to \begin{equation}\label{eq:pompeiu}
	\mathcal{P}_Ef(R) = \int_{R(E)}f(x,y)\, dx\, dy, 
\end{equation}
where $R$ varies in the (orientation-preserving) Euclidean motion group $M^+(2)$ and $E$ is a fixed bounded subset of the plane of positive area. In other words, given $E$ one asks whether knowledge of all integrals of a locally integrable function $f$ over rotated and translated copies of $E$ suffices to uniquely determine $f$. The answer is negative when $E$ is a disc, and is conjectured to be positive for every other convex (or more generally Jordan) domain. Roughly speaking, one believes that the only obstruction to injectivity of $\mathcal{P}_E$ is rotation-invariance of $E$, which makes the family $\{R(E)\colon\, R\in M^+(2)\}$ two-dimensional when $E$ is a disc, while the family is generally three-dimensional. See, e.g., \cite{zalcman} for an introduction to this topic. As may be glimpsed from the last section of that paper, making progress in these matters requires quite sophisticated analytical methods.

For our purposes, it is important to stress the fact that Pompeiu's problem becomes significantly easier if one restricts consideration to functions $f$ in $L^\infty(\R^2)$. In fact, using the fact that $L^\infty$ is the Banach dual of $L^1$, one may easily see that the transform $\mathcal{P}_E$ is injective on $L^\infty$ if and only if the linear span of \[
\{1_{R(E)}\colon\, R\in M^+(2)\}
\]
is dense in $L^1(\R^2)$ (here $1_F$ is the indicator function of $F\subseteq \R^2$). The classical Wiener's Tauberian Theorem shows that this is in turn equivalent to the following condition: the Fourier transform  \begin{equation}\label{eq:fourier}
	\widehat{1_E}(\xi) = \int_E e^{-2\pi i \xi\cdot x}\, dx \qquad (\xi\in \R^2)
\end{equation} of the indicator of $E$ \emph{does not vanish identically on any circle centered at the origin}. In short, \begin{equation}\label{eq:fourier_pompeiu}
	\{r\in \R\colon\ \widehat{1_E}(r\omega)=0\quad \forall \omega\in S^1\}=\varnothing.
\end{equation} See Section 6 of \cite{brown_schreiber_taylor} for a proof of this fact (which, anyway, will not be used in what follows). E.g., using this neat Fourier analytic characterization, one may rapidly check that $\mathcal{P}_E$ is injective on $L^\infty(\R^2)$ when $E$ is a square and $\mathcal{P}_E$ is not injective on $L^\infty(\R^2)$ when $E$ is a disc, just by computing the Fourier transforms of the indicators of these sets. Moreover, the characterization \eqref{eq:fourier_pompeiu} is valid for all $E$ for which the Fourier integral \eqref{eq:fourier} is absolutely convergent, that is, whenever $E$ has finite area (notice that, if $E$ is unbounded, the operator $\mathcal{P}_E$ is not well-defined on locally integrable functions, so the original formulation of Pompeiu's problem does not make sense). 

The ideal triangle transform is the hyperbolic analogue of \eqref{eq:pompeiu} where $E$ is an ideal triangle (which has finite hyperbolic area). Our proof of Theorem \ref{thm:pompeiu} mirrors the Euclidean argument discussed in the previous paragraph, exploiting a Wiener's Tauberian Theorem on the hyperbolic plane (discussed in the next section), where the role of the customary Euclidean Fourier transform is taken by the Helgason--Fourier transform. 

\subsection{Fourier analysis and Wiener's Tauberian theorem on the hyperbolic plane}

Let us begin with a brief sketch of the basics of Fourier analysis on $\mathbb{H}^2$. See \cite{helgason_groups} for a more thorough treatment. The usual formula for the Fourier transform on the plane $\R^2$, \[
\widehat{f}(\xi)=\int_{\R^2}f(x)e^{-ix\cdot \xi}\, dx 
\] admits the following geometric interpretation. Writing $\xi=\lambda \omega$, where $\lambda\in \R$ and $\omega\in S^1$, we may interpret the quantity $x\cdot \xi=\lambda\, x\cdot \omega$ as a multiple (given by the parameter $\lambda$) of the signed distance from the origin $0\in \R^2$ to the line through $x$ perpendicular to $\omega$. Thus, $e^{i\lambda x\cdot \omega}$ is the "plane wave" of signed frequency $\lambda$, propagating along a family of geodesics of direction $\omega\in S^1$.

On the hyperbolic plane $\mathbb{H}^2$, one similarly considers plane waves propagating along the family of geodesics meeting at a point $b$ in the ideal boundary $\partial \mathbb{H}^2$. The correct analogue of $x\cdot \omega$ is then the signed distance $\left\langle z, b\right\rangle$ from a fixed origin in $\mathbb{H}^2$ to the curve through $z\in \mathbb{H}^2$ that is perpendicular to such a family of geodesics. The latter curve is of course the horocycle through $z$ tangent to $b$, and $\left\langle z,b\right\rangle$ is the \emph{Busemann function} associated to $b$ (based at the given origin). We take the sign of $\left\langle z, b\right\rangle$ to be positive if the origin lies outside the horoball bounded by the horocycle.

Working in the upper half-plane model $\mathcal{H}^2=\{z=x+iy\in \C\colon\, y>0\}$ of the hyperbolic plane, with ideal boundary $\partial \mathcal{H}^2=\R\cup\{\infty\}$ and origin $i\in \mathcal{H}^2$, a standard computation gives \[
\left\langle z, b\right\rangle=\log \frac{(b^2+1)y}{(x-b)^2+y^2}
\]
for $b\in \R$, and \[
\left\langle z, \infty\right\rangle=\log y.
\] The analogue of the Euclidean plane wave $e^{i\lambda\,  x\cdot \omega}$ is the function \begin{equation}\label{eq:hyp_plane_wave}
	e^{(i\lambda+1) \left\langle z, b\right\rangle} = \begin{cases}\left(\frac{(b^2+1)y}{(x-b)^2+y^2}\right)^{i\lambda+1}\quad &b\in \R\\
		y^{i\lambda+1}\quad &b=\infty
	\end{cases}
\end{equation} and the Helgason--Fourier transform of $f\in L^1(\mathcal{H}^2):= L^1(\mathcal{H}^2, dA)$, where $dA(z)=\frac{dx\, dy}{y^2}$ is the hyperbolic area, is defined as
\begin{eqnarray}
	\notag\widetilde{f}(\lambda, b) &=& \int_{\mathcal{H}^2} f(z) e^{(-i\lambda+1) \left\langle z, b\right\rangle} \, dA(z) \\
	\label{eq:fourier_helgason}&=&  \begin{cases}(b^2+1)^{-i\lambda+1}\int_{\mathcal{H}^2} f(x+iy) \left(\frac{y}{(x-b)^2+y^2}\right)^{-i\lambda+1} \, \frac{dx\, dy}{y^2}\qquad &(b\in \R)\\
		\int_{\mathcal{H}^2} f(x+iy) y^{-i\lambda+1} \, \frac{dx\, dy}{y^2}\qquad &(b=\infty)
	\end{cases}
\end{eqnarray}
The factor $+1$ in the exponent in formulas \eqref{eq:hyp_plane_wave} and \eqref{eq:fourier_helgason} is a manifestation of a remarkable difference between Fourier analysis in the Euclidean and in the hyperbolic settings. We refer to Helgason's book \cite{helgason_groups} for more on this point, limiting ourselves to the following remark. While the Euclidean Fourier transform $\widehat{f}(\lambda\omega)$ of an integrable function $f$ is only defined for real values of the frequency parameter $\lambda$, the Helgason--Fourier transform $\widetilde{f}(\lambda, b)$ of an integrable function $f$ converges, for almost every value of $b$, for frequencies $\lambda$ in the strip \[
S=\{\lambda\in \C\colon\, |\mathrm{Im}(\lambda)|\leq 1\}
\] and defines a holomorphic function on the interior of this strip that extends continuously up to its boundary. Under the stronger integrability assumption \begin{equation}\label{eq:extra_integrability}
	\int_{\mathcal{H}^2} |f(z)| e^{\epsilon d(z,i)}\, dA(z)<\infty, 
\end{equation}
where $\epsilon>0$ and $d$ is the hyperbolic distance, the Helgason--Fourier transform extends holomorphically to the interior, and continuously up to the boundary, of the even larger strip \[
S_\epsilon=\{\lambda\in \C\colon\, |\mathrm{Im}(\lambda)|\leq 1+\epsilon\}.
\] This is easily seen from the formula of the Fourier--Helgason transform and the fact that $|\left\langle z, b\right\rangle|\leq d(z,i)$. 

We can now state a version of the Wiener's Tauberian Theorem for the hyperbolic plane due to Mohanty, Ray, Sarkar and Sitaram, that will be sufficient for our analysis of injectivity of the ideal triangle transform. The statement below is Theorem 5.5 of \cite{mohanty_etal} (cf.~also \cite[Theorem 2.8]{helgason_noneuclidean}), which builds on various antecedents, see \cite{benyamini_weit, bennatan_etal, sarkar}. 

\begin{thm}\label{thm:wiener_hyperbolic}
	Let $f$ be a (measurable) function on $\mathcal{H}^2$ that satisfies the integrability assumption $\int_{\mathcal{H}^2} |f(z)| e^{\epsilon d(z,i)}\, dA(z)<\infty$ for some $\epsilon>0$, and is not a.e.~equal to a real-analytic function. Assume that \[
	\{\lambda\in S_\epsilon\colon\, \widetilde{f}(\lambda, b)=0\quad \text{for a.e.}\quad b\in \partial \mathcal{H}^2\}=\varnothing. 
	\]
	Then the linear span of \[
	\{f\circ g^{-1}\colon\, g\in \mathrm{Isom}^+(\mathcal{H}^2)\} 
	\] is dense in $L^1(\mathcal{H}^2)$.
\end{thm}

The hypothesis that $f$ be non-real-analytic ensures that the Helgason--Fourier transform $\widetilde{f}(\lambda, b)$ does not decay too fast on the line $\mathrm{Im}(\lambda)=0$, a known necessary condition for a theorem like the above to hold true. See \cite[p.~680]{benyamini_weit} for a discussion of this point. At any rate, we will apply Theorem \ref{thm:wiener_hyperbolic} to the indicator of an ideal triangle, for which the assumption is clearly satisfied. Now that all the ingredients are in place, we can proceed with the proof of Theorem \ref{thm:pompeiu}, or equivalently Theorem \ref{thm:pompeiu_0}. 

\subsection{Proof of Theorem \ref{thm:pompeiu}} Fix the ideal triangle \[
T_0=\{z=x+iy\in \mathcal{H}^2\colon\, -1<x<1, \, y>\sqrt{1-x^2}\}.
\]
Let $f\in L^\infty(\mathcal{H}^2)$ be such that $\mathcal{I}f(g)=0$ for all $g\in \mathrm{Isom}^+(\mathcal{H}^2)$. As anticipated in Section \ref{sec:preliminary_pompeiu}, by the duality between $L^1$ and $L^\infty$, we may view $f$ as a continuous linear functional on $L^1(\mathcal{H}^2)$ that vanishes on the set \[
\{1_{g(T_0)}=1_{T_0}\circ g^{-1}\colon\, g\in  \mathrm{Isom}^+(\mathcal{H}^2)\}.
\] In order to prove that $f=0$, it is therefore enough to show that the linear span of the above set is dense in $L^1(\mathcal{H}^2)$. By Theorem \ref{thm:wiener_hyperbolic}, the proof boils down to verifying the following two facts: \begin{enumerate}
	\item[(I)] the integral $\int_{T_0} e^{\epsilon d(z,i)}\, dA(z)$ is finite for some positive $\epsilon$; 
	\item[(II)] there exists no $\lambda$ in the strip $S_\epsilon$ for which the Fourier--Helgason transform $\widetilde{1_{T_0}}(\lambda, b)$ vanishes for a.e.~$b\in \partial\mathcal{H}^2$. 
\end{enumerate}

Let us check (I). By the symmetry of the integral under isometries permuting the vertices of $T_0$, it is enough to consider the contribution of the region $\{y\geq 1\}$. Since \[d(x+iy,i)\leq d(x+iy, iy)+d(iy, i)=1+\log y\] for $x+iy\in T_0\cap \{y\geq 1\}$, we have \[
\int_{T_0\cap \{y\geq 1\}} e^{\epsilon d(z,i)}\, dA(z)\leq 2e^\epsilon \int_1^\infty y^{\epsilon-2}\, dy, 
\]
which is finite for all $\epsilon<1$. \newline 

Let us turn to fact (II). Define \[
F(s, b)=\int_{T_0} \left(\frac{y}{(x-b)^2+y^2}\right)^s\, \frac{dx\, dy}{y^2}, 
\]
where $s$ is a complex parameter and $b$ is real. Notice that \begin{equation}\label{eq:fourier_helgason_F}
	\widetilde{1_{T_0}}(\lambda, b)=(b^2+1)^{-i\lambda+1}F(-i\lambda+1, b)
\end{equation} for all $\lambda\in S_\epsilon$ and $b\in \R$ for which the Fourier--Helgason integral converges. \newline 

Denote by $P$ the half-plane where $\mathrm{Re}(s)>-1$. 

\begin{prop}\label{prp:integrability} The integral $F(s,b)$ is absolutely convergent for $b\in \R\setminus \{-1,+1\}$ and $s$ in the half-plane $P$. For each $s\in P$ the function $F(s,\cdot)$ is of class $C^2$ (away from $-1$ and $+1$), and for each $b\in \R\setminus \{-1,+1\}$ and $j\leq 2$, the function $\frac{\partial^j F}{\partial b^j}(\cdot, b)$ is holomorphic on $P$. The derivatives can be computed differentiating under the integral sign. In particular, \begin{equation}\label{eq:second_derivative_F}
		\frac{\partial^2 F}{\partial b^2}(s,0) = \int_{T_0}\left(\frac{y}{x^2+y^2} \right)^s\frac{(4s^2+2s)x^2-2sy^2}{(x^2+y^2)^2}\, \frac{dx\, dy}{y^2} 
	\end{equation} for all $s\in P$. 
\end{prop}

By similar arguments to those used below, one may in fact check that $F(s,b)$ is infinitely differentiable in $P\times (\R\setminus \{-1,1\})$, but we will not need this. 

\begin{proof}
	For each fixed $(x,y)\in T_0$, the integrand $\left(\frac{y}{(x-b)^2+y^2}\right)^sy^{-2}$ is jointly smooth in $(b,s)\in \R\times \C$ and holomorphic in $s$. In order to justify the statement by a standard dominated convergence argument, all we need to show is that  \begin{equation}\label{eq:dominated_convergence}
		\max_{(b,s)\in K,\,  j\leq 2} \left|\left(\frac{\partial}{\partial b}\right)^j\left(\frac{y}{(x-b)^2+y^2} \right)^s\right|\in L^1\left(T_0, \frac{dx\, dy}{y^2}\right)
	\end{equation} for each compact subset $K$ of $P \times (\R\setminus \{-1,+1\})$. This shows that $\frac{\partial^j F}{\partial b^j}$ ($j\leq 2$) is jointly continuous, while holomorphicity follows as usual from Morera's Theorem (via Fubini's Theorem and the holomorphicity of the integrand). \newline 
	
	We use a couple of elementary bounds. Let $b\in \R\setminus \{-1,1\}$. If $(x,y)\in T_0$, then \[
	(x-b)^2+y^2\geq (x-b)^2+1-x^2 = 1-2xb+b^2.
	\] Obviously, $\min_{x\in [-1,1]}1-2xb+b^2$ is achieved at one of the endpoints, so \begin{equation}\label{eq:elementary_bound}
		(x-b)^2+y^2\geq c_b:=\min\{(1-b)^2, (1+b)^2\} \qquad \forall (x,y)\in T_0. 
	\end{equation}
	Let $T_\mathrm{up}:=T_0\cap \{y\geq 1\}$ and $T_\mathrm{low}:=T_0\cap \{y\leq 1\}$. By \eqref{eq:elementary_bound}, for $(x,y)\in T_\mathrm{low}$, we get \[
	c_b\leq (x-b)^2+y^2\leq (1+|b|)^2+1=:k_b, 
	\]
	while for $(x,y)\in T_\mathrm{up}$, we have \[
	y^2\leq (x-b)^2+y^2\leq (1+|b|)^2+y^2\leq k_by^2.
	\] Let $\tilde{c}_b:=\min\{1,c_b\}$. Combining the above estimates, we get  \[
	\frac{1}{k_b}\min\{y, y^{-1}\}\leq\frac{y}{(x-b)^2+y^2} \leq \frac{1}{\tilde{c}_b} \min\{y, y^{-1}\}\qquad \forall (x,y)\in T_0. 
	\]
	If $s\in \C$, \begin{equation}\label{eq:elementary_bound_2}
		\left|\left(\frac{y}{(x-b)^2+y^2} \right)^s\right|\leq C_{b,s}\min\{y, y^{-1}\}^{\mathrm{Re}(s)}\qquad \forall (x,y)\in T_0, 
	\end{equation}
	where $C_{b,s}=\max\{k_b^{-\mathrm{Re}(s)}, \tilde{c}_b^{-\mathrm{Re}(s)}\}$. All that matters for us is that the quantity $C_{b,s}$ remains bounded when $(b,s)$ varies in a compact subset of $(\R\setminus \{-1,+1\})\times \C$. We say that such a quantity is \emph{admissible}.
	
	Next, we compute 
	\[
	\frac{\partial}{\partial b}\left(\frac{y}{(x-b)^2+y^2} \right)^s = s\left(\frac{y}{(x-b)^2+y^2} \right)^s\frac{2(x-b)}{(x-b)^2+y^2}.
	\]
	By \eqref{eq:elementary_bound} and \eqref{eq:elementary_bound_2},  \begin{equation}\label{eq:elementary_bound_3}
		\left|\frac{\partial}{\partial b}\left(\frac{y}{(x-b)^2+y^2} \right)^s\right|\leq C_{b,s}'\min\{y, y^{-1}\}^{\mathrm{Re}(s)}\qquad \forall (x,y)\in T_0, 
	\end{equation}
	where $C_{b,s}':=2|s|C_{b,s}c_b^{-1}(1+|b|)$ is admissible. Similarly, \begin{eqnarray*}
		&&\left(\frac{\partial}{\partial b}\right)^2\left(\frac{y}{(x-b)^2+y^2} \right)^s \\
		&=& s^2\left(\frac{y}{(x-b)^2+y^2} \right)^s\frac{4(x-b)^2}{((x-b)^2+y^2)^2}\\
		&+&s\left(\frac{y}{(x-b)^2+y^2} \right)^s\left(-\frac{2}{(x-b)^2+y^2}+\frac{4(x-b)^2}{((x-b)^2+y^2)^2}\right).
	\end{eqnarray*}
	and \begin{equation}\label{eq:elementary_bound_4}
		\left|\left(\frac{\partial}{\partial b}\right)^2\left(\frac{y}{(x-b)^2+y^2} \right)^s\right|\leq C_{b,s}''\min\{y, y^{-1}\}^{\mathrm{Re}(s)}\qquad \forall (x,y)\in T_0, 
	\end{equation}
	where $C_{b,s}'':=(4|s|^2(1+|b|)^2c_b^{-2}+2|s|c_b^{-1}+4|s|(1+|b|)^2c_b^{-2})C_{b,s}$ is also admissible. \newline

	We can finally look at the integral $F(b,s)$. Fix $\delta>0$, $R>0$, and assume that $\mathrm{Re}(s)\geq -1+\delta$ and $|s|\leq R$. By \eqref{eq:elementary_bound_2}, \eqref{eq:elementary_bound_3} and \eqref{eq:elementary_bound_4}, the integrand and its first two derivatives in $b$ are bounded in modulus by $C_{b,\delta, R}\min\{y, y^{-1}\}^{-1+\delta}y^{-2}$, where $C_{b, R}$ remains bounded if $b$ stays away from $+1$ and $-1$. We have \begin{eqnarray*}
		&& \int_{T_0} 	\min\{y, y^{-1}\}^{-1+\delta}y^{-2}\, dx\, dy \\
		&=& \int_{-1}^{+1}\left(\int_{\sqrt{1-x^2}}^1  y^{-3+\delta}\, dy+\int_{1}^{+\infty}  y^{-1-\delta}\, dy \right)\, dx\\
		&=& \int_{-1}^{+1} (2-\delta)^{-1}[(1-x^2)^{-1+\frac{\delta}{2}}-1]\, dx + 2\delta^{-1}.
	\end{eqnarray*}
	Since $1-x^2$ vanishes with non-zero derivative at $-1$ and $+1$, the above quantity is finite. This completes the proof of \eqref{eq:dominated_convergence}. Identity \eqref{eq:second_derivative_F} is obtained evaluating the formula for $\frac{\partial^2F}{\partial b^2}$ at $b=0$. 
\end{proof} 

The proposition just proved shows that $F(-i\lambda+1,b)$, and hence the Fourier--Helgason transform of the indicator of the ideal triangle, is convergent for $\mathrm{Im}(\lambda)>-2$, which includes the interior of the strip $S_1$, as expected from the extra-integrability of the indicator (fact (I)). The next proposition gives explicit representations of $F(s,0)$ and $\frac{\partial^2F}{\partial b^2}(s,0)$ as ratios of Gamma functions. Fact (II) will be deduced from these formulas. 

\begin{prop}\label{prp:ratio_gamma}
	If $s\in P$, then \begin{eqnarray}
		\label{eq:F1}F(s,0)&=&\frac{1-2^{1-s}}{2}\frac{\sqrt{\pi}\, \Gamma\left(\frac{s-1}{2}\right)}{\Gamma\left(\frac{s}{2}+1\right)}\\
		\label{eq:F2}\frac{\partial^2 F}{\partial b^2}(s,0) &=&2(1-(s+1)2^{-s})\frac{\sqrt{\pi}\, \Gamma\left(\frac{s-1}{2}\right)}{\Gamma\left(\frac{s}{2}\right)}. 
	\end{eqnarray}
\end{prop}

\begin{proof} In terms of polar coordinates $(r,\theta)$ such that $x=r\cos\theta$ and $y=r\sin\theta$, the right half of the ideal triangle $T_0^+=T_0\cap \{\mathrm{Re}(s)\geq 0\}$ is defined by the conditions \[
	0< \theta\leq \pi/2,\qquad 1< r< \frac{1}{\cos\theta}.
	\]
	Since the integrand of $F(s,0)$ is invariant under the symmetry $(x,y)\mapsto (-x,y)$, for $s\neq 0$ with $\mathrm{Re}(s)>-1$ we have \begin{eqnarray}
		\notag F(s,0)&=&2\int_{T_0^+} \left(\frac{y}{x^2+y^2}\right)^s\, \frac{dx\, dy}{y^2}\\
		\notag &=&2\int_0^{\pi/2}\int_1^{1/\cos\theta} r^{-s-1}(\sin\theta)^{s-2}\, dr \, d\theta\\
		\label{eq:F}&=&\frac{2}{s}\int_0^{\pi/2}(1-(\cos\theta)^s)(\sin\theta)^{s-2} \, d\theta.
	\end{eqnarray}
	Notice that $1-(\cos\theta)^s$ vanishes (at least) quadratically at $\theta=0$ (for each $s\in \C$), so the integral in the last line is absolutely convergent for $s\in P$ (as it should in view of Proposition \ref{prp:integrability}) and defines a holomorphic function there. This holomorphic function has a zero at $s=0$, which cancels the singularity $\frac{2}{s}$, consistently with the holomorphicity of $F(s,0)$ on the half-plane $P$. 
	
	We now use the Gamma (or Beta) function identity (cf.~identity (1.1.21) in \cite{andrews_book})
	\begin{equation}\label{eq:beta}
		\int_0^{\pi/2}(\sin\theta)^{s_1}(\cos\theta)^{s_2}\, d\theta = \frac{\Gamma\left(\frac{s_1+1}{2}\right)\Gamma\left(\frac{s_2+1}{2}\right)}{2\Gamma\left(\frac{s_1+s_2}{2}+1\right)}\qquad (\mathrm{Re}(s_1),\, \mathrm{Re}(s_2)>-1). \end{equation} We get \begin{equation}\label{eq:gamma_identity}
		\int_0^{\pi/2}(1-(\cos\theta)^s)(\sin\theta)^{s-2} \, d\theta = \frac{\Gamma\left(\frac{s-1}{2}\right)}{2}\left(\frac{\sqrt{\pi}}{\Gamma\left(\frac{s}{2}\right)}-\frac{\Gamma\left(\frac{s+1}{2}\right)}{\Gamma(s)}\right), 
	\end{equation}
	where we also used the identity $\Gamma\left(\frac{1}{2}\right)=\sqrt{\pi}$.  Identity \eqref{eq:gamma_identity} is first established for $\mathrm{Re}(s)>1$ and then extended by analytic continuation to the half-plane $P$, where the left hand side is holomorphic. 
	
	We next employ the Legendre duplication formula (Theorem 1.5.1 of \cite{andrews_book})\[
	\Gamma(z)\Gamma\left(z+\frac{1}{2}\right)=2^{1-2z}\sqrt{\pi}\, \Gamma(2z), 
	\] 
	which evaluated at $z=\frac{s}{2}$ and rearranged, may be rewritten in the form \begin{equation}\label{eq:legendre}
		\frac{\Gamma\left(\frac{s+1}{2}\right)}{\Gamma(s)}=2^{1-s}\frac{\sqrt{\pi}}{\Gamma\left(\frac{s}{2}\right)}.
	\end{equation}
	Plugging this into \eqref{eq:gamma_identity}, we finally obtain  \begin{equation}\label{eq:gamma_ratio_1}
		\int_0^{\pi/2}(1-(\cos\theta)^s)(\sin\theta)^{s-2} \, d\theta =\frac{1-2^{1-s}}{2}\frac{\sqrt{\pi}\, \Gamma\left(\frac{s-1}{2}\right)}{\Gamma\left(\frac{s}{2}\right)}.
	\end{equation}
	By \eqref{eq:F}, \eqref{eq:gamma_ratio_1}, and the functional equation for the Gamma function, we finally obtain \eqref{eq:F1}. 
	
	We now treat similarly the second derivative of $F$ with respect to $b$. By the identity of Proposition \ref{prp:integrability}, \begin{eqnarray}
		\notag	\frac{\partial^2 F}{\partial b^2}(s,0) &=& 2\int_{T_0^+}\left(\frac{y}{x^2+y^2} \right)^s\frac{(4s^2+2s)x^2-2sy^2}{(x^2+y^2)^2}\, \frac{dx\, dy}{y^2} \\
		\notag	&=&2\int_0^{\pi/2}\int_1^{1/\cos\theta} r^{-s-3}(\sin\theta)^{s-2}((4s^2+2s)(\cos\theta)^2-2s(\sin\theta)^2)\, dr \, d\theta\\
		\label{eq:second_derivative}	&=&\frac{4}{s+2}\int_0^{\pi/2}(1-(\cos\theta)^{s+2})(\sin\theta)^{s-2}((2s^2+s)-(2s^2+2s)(\sin\theta)^2)\, d\theta.
	\end{eqnarray}
	We now need two variants of \eqref{eq:gamma_ratio_1}. The first follows from \eqref{eq:beta}, the functional equation, and the duplication formula \eqref{eq:legendre}, 
	\begin{eqnarray}
		\notag	\int_0^{\pi/2}(1-(\cos\theta)^{s+2})(\sin\theta)^{s-2} \, d\theta &=&\frac{\Gamma\left(\frac{s-1}{2}\right)}{2}\left(\frac{\sqrt{\pi}}{\Gamma\left(\frac{s}{2}\right)}-\frac{\Gamma\left(\frac{s+3}{2}\right)}{\Gamma(s+1)}\right)\\ \notag &=&\frac{\Gamma\left(\frac{s-1}{2}\right)}{2}\left(\frac{\sqrt{\pi}}{\Gamma\left(\frac{s}{2}\right)}-\frac{s+1}{2s}\frac{\Gamma\left(\frac{s+1}{2}\right)}{\Gamma(s)}\right)\\ \label{eq:gamma_ratio_2}&=&\frac{1}{2}\left(1-\frac{s+1}{s}2^{-s}\right)\frac{\sqrt{\pi}\, \Gamma\left(\frac{s-1}{2}\right)}{\Gamma\left(\frac{s}{2}\right)}.
	\end{eqnarray} Evaluating \eqref{eq:gamma_ratio_1} at $s+2$ and using once more the functional equation, we get the second one: \begin{equation}\label{eq:gamma_ratio_3}
		\int_0^{\pi/2}(1-(\cos\theta)^{s+2})(\sin\theta)^s \, d\theta =\frac{1-2^{-1-s}}{2}\frac{s-1}{s}\frac{\sqrt{\pi}\, \Gamma\left(\frac{s-1}{2}\right)}{\Gamma\left(\frac{s}{2}\right)}.
	\end{equation}
	Plugging \eqref{eq:gamma_ratio_2} and \eqref{eq:gamma_ratio_3} into \eqref{eq:second_derivative}, we finally obtain \begin{eqnarray*}
		\notag\frac{\partial^2 F}{\partial b^2}(s,0) &=&\frac{2}{s+2}\left((2s^2+s)\left(1-\frac{s+1}{s}2^{-s}\right)-(2s^2+2s)\frac{s-1}{s}(1-2^{-1-s})\right)\frac{\sqrt{\pi}\, \Gamma\left(\frac{s-1}{2}\right)}{\Gamma\left(\frac{s}{2}\right)}\\
		&=&2(1-(s+1)2^{-s})\frac{\sqrt{\pi}\, \Gamma\left(\frac{s-1}{2}\right)}{\Gamma\left(\frac{s}{2}\right)}, 
	\end{eqnarray*}
	which is \eqref{eq:F2}. 
\end{proof}

We are now in a position to verify fact (II) (for any $\epsilon<1$), thus completing the proof of Theorem \ref{thm:pompeiu_0}. In view of \eqref{eq:fourier_helgason_F}, this amounts to checking that the function $b\mapsto F(s,b)$ is not a.e.~zero for any $s$ such that $\mathrm{Re}(s)\in (-1,3)$. By Proposition \ref{prp:ratio_gamma}, it is more than enough to show that the two conditions \begin{equation}\label{eq:system_F}
	\begin{cases} (1-2^{1-s})\frac{\Gamma\left(\frac{s-1}{2}\right)}{\Gamma\left(\frac{s}{2}+1\right)}=0\\
		(1-(s+1)2^{-s})\frac{\Gamma\left(\frac{s-1}{2}\right)}{\Gamma\left(\frac{s}{2}\right)}=0
	\end{cases}
\end{equation}
are never simultaneously satisfied in $P$. Now the Gamma function has no zeros, the denominator $\Gamma\left(\frac{s}{2}+1\right)$ has no poles in $P$, and the only pole of $\Gamma\left(\frac{s}{2}\right)$ is canceled by a zero of the numerator; therefore, a solution $s$ to the above system must satisfy  \[
\begin{cases} 1-2^{1-s}=0\\
	1-(s+1)2^{-s}=0
\end{cases}
\] which means $s=1$. Since the numerator $\Gamma\left(\frac{s-1}{2}\right)$ has a simple pole at $s=1$ and the factors $1-2^{1-s}$ and $1-(s+1)2^{-s}$ have simple zeros at $s=1$, this value is neither a zero of $F(s,0)$ nor of $\frac{\partial^2F}{\partial b^2}(s,0)$. Thus, \eqref{eq:system_F} has no solutions, and the proof is complete. 

\section{Proof of Theorems~\ref{thm:pompeiu_1} and~\ref{main:thm}}\label{section:proof}
As anticipated in the introduction, it is easy to deduce Theorem~\ref{thm:pompeiu_1} from Theorem~\ref{thm:pompeiu_0}. Indeed,
let $\widetilde{\omega}\in Z\Omega^2_b(\mathbb{H}^n)$ be such that
$$
\int_T \widetilde{\omega}=0
$$
for every ideal triangle $T\subseteq \Hn$. Take $p\in\mathbb{H}^n$ and an orthonormal frame $(e_1,e_2)$ in the tangent space $T_p \Hn$.
Let $K\subseteq \Hn$ be the hyperbolic plane containing $p$ and such that the tangent plane $T_p K$ is generated by $e_1$ and $e_2$. If $i\colon K\to \Hn$
is the inclusion, then we have $i^*\widetilde{\omega}=f\, dA$ for some $f\in L^\infty (K)$, where $dA$ is the area form on $K$ (in fact, $f$ is even smooth). 
Our assumption
implies in particular that $$\int_T f\, dA=\int_T \widetilde{\omega}=0$$
for every ideal triangle $T\subseteq K$.
Since $K$ is just an isometric copy of $\Hp$, Theorem~\ref{thm:pompeiu_0} now implies that $f=0$ almost everywhere on $K$ (hence, everywhere on $K$, since $f$ is smooth).
In particular, $(i^*\widetilde{\omega})_p=0$, and $\widetilde{\omega}_p(e_1,e_2)=0$. Due to the arbitrariness of $p$ and of $e_1,e_2$, this gives in turn $\widetilde{\omega}=0$. This proves Theorem~\ref{thm:pompeiu_1}.

\bigskip

Finally, let $M=\Hn/\Gamma$ be a hyperbolic manifold of the first kind, let $\psi\colon Z\Omega^2_b(M)\to H^2_b(M)$ be the map described in the introduction, and take
$\omega\in \ker \psi$. Theorem~\ref{annullamento:thm} implies that $
\int_T \widetilde{\omega}=0
$
for every ideal triangle $T\subseteq \Hn$, hence $\omega=0$ by Theorem~\ref{thm:pompeiu_1}. This concludes the proof of Theorem~\ref{main:thm}.

\bibliography{math_papers_final}
\bibliographystyle{alpha}

\end{document}